\documentclass[12pt,a4paper]{amsart}

\usepackage{amssymb}
\usepackage{amsfonts}
\usepackage{amsmath}
\usepackage{amsthm}
\usepackage{latexsym}
\usepackage{enumerate}

\newcommand{\renyi}{R\'enyi}
\newcommand{\renyis}{\renyi{ }}

\newcommand{\bs}{\boldsymbol}

\newcommand{\ef}{\mathbb{F}}

\newcommand{\efq}{\ef_q}

\newcommand{\bk}{\bs{k}}
\newcommand{\bK}{\bs{K}}

\newcommand{\bx}{\bs{x}}

\newcommand{\bA}{\bs{A}}
\newcommand{\ba}{\bs{a}}

\newcommand{\xkys}{X_{k,y}}

\newcommand{\gk}{g_{\bk}}
\newcommand{\gks}{g_k}
\newcommand{\gK}{g_{\bK}}

\DeclareMathOperator{\pr}{Pr}

\DeclareMathOperator{\cp}{cp}
\DeclareMathOperator{\charac}{char}
\DeclareMathOperator{\weight}{w}
\DeclareMathOperator{\dham}{d_H}

\newtheorem{proposition}{Proposition}

\newtheorem{lemma}[proposition]{Lemma}
\newtheorem{theorem}[proposition]{Theorem}

\newtheorem*{remark}{Remark}

\title[A lower bound on the average entropy]{A lower bound on the
average entropy of a function Determined up to a diagonal linear Map
on $\mathbb{F}_q^n$}   
\author[Yaron Shany]{Yaron Shany}
\address{25B Sirkin St. Kfar Saba, Israel}
\email{yaron.shany@gmail.com}

\author[Ram Zamir]{Ram Zamir}
\address{Department of EE-Systems, Tel Aviv University, Tel Aviv, Israel}
\email{zamir@eng.tau.ac.il}

\begin{document}

\begin{abstract}
In this note, it is shown that if $f\colon\efq^n\to\efq^n$ is any  
function and $\bA=(A_1,\ldots, A_n)$ is uniformly distributed over
$\efq^n$, then the average over $(k_1,\ldots,k_n)\in \efq^n$ of the
\renyis (and hence, of the Shannon) entropy of
$f(\bA)+(k_1A_1,\ldots,k_nA_n)$ is at least about $\log_2(q^n)-n$
bits. In fact, it is shown that the average collision probability of
$f(\bA)+(k_1A_1,\ldots,k_nA_n)$ is at most about $2^n/q^n$. 
\end{abstract}

\maketitle
\section{Introduction}

Suppose that $f\colon\efq\to \efq$ is an arbitrary function (where $q$
is a prime power and $\efq$ is the finite field of $q$ elements). Let
$A$ be a random variable uniformly distributed over $\efq$. Clearly,
$f(A)$ may be far from uniform, while $kA$ is uniform for all
$k\in\efq^*$. Is
$f(A)+kA$ nearly uniform for most values of $k\in \efq$?  More
generally, given a positive integer $n$, for
an arbitrary $f\colon \efq^n\to\efq^n$ and for $\bA$ uniformly
distributed over $\efq^n$, is\footnote{Throughout, we write $x_i$ for the
$i$th coordinate of a vector $\bx$. Also, for a function $f$ with codomain
$\efq^n$, we will write $f_i$ for the $i$th component of $f$
(post-composition of $f$ with the $i$th projection)} $f(\bA)+(k_1A_1,\ldots,k_nA_n)$
nearly uniform for most values of $\bk\in\efq^n$? 

Recall that the {\it Shannon entropy} $H(B)$ of a random variable $B$
taking values in a finite set $S$ is defined by\footnote{From this
point on, all logarithms are to the base of 2.}
$H(B):=-\sum_{s:\pr(B=s)\neq 0} \pr(B=s)\cdot\log(\pr(B=s))$, while
the {\it collision probability} of $B$, $\cp(B)$, is defined by 
$\cp(B):=\sum_{s\in S} \pr(B=s)^2=\pr(B=B')$, where $B'$ 
is an independent copy of $B$. The {\it \renyi} entropy of $B$,
$H_2(B)$, is defined by $H_2(B):=-\log(\cp(B))$.  A straightforward
application of Jensen's inequality shows that $H_2(B)\leq H(B)$.

Since both the \renyis entropy and the Shannon entropy measure
randomness (where for both entropies the maximum possible value of
$\log(|S|)$ is equivalent to having uniform distribution, and the
minimum possible value of 0 is equivalent to being deterministic), a
possible formal phrasing of the above question on
$f(\bA)+(k_1A_1,\ldots,k_nA_n)$ is:  How much smaller
than $\log(q^n)$ might the average over $\bk$ of the \renyis (or
Shannon) entropy be? 

The collision probability itself is yet an additional measure of
randomness, where the minimum collision probability of $1/|S|$ is
equivalent to having uniform distribution and the maximum possible
collision probability of 1 is equivalent to being deterministic.  So,
another possible formal phrasing of the question on
$f(\bA)+(k_1A_1,\ldots,k_nA_n)$ is:  How much larger
than $1/q^n$ might the average over $\bk$ of the collision probability
be? 

The  main motivation for this question is a
certain  side-information problem in information theory \cite{Zamir}.
Several neighboring questions were considered 
in the literature.  For example, the case $n=1$ of Theorem
\ref{thm:main} ahead extends Lemma 21 of \cite{KLSS}, stating that
for any $f\colon\efq\to\efq$ there exists $k\in \efq$ for which
$|\{f(x)+kx|x\in\efq\}|>q/2$.\footnote{It should be noted
that in this case ($n=1$), the result follows immediately from the
Leftover Hash Lemma as described, e.g., in Lemma 7.1 of \cite{Sti}, or
in Theorem 8 of \cite{CF}.} The same case of Theorem
\ref{thm:main} ahead also extends the main theorem of \cite{Car}, which
states that the average over $k\in\efq$  of $|\{f(x)+kx|x\in\efq\}|$
(for $f$ a polynomial of degree $<\charac(\efq)$) is at least
$q/(2-1/q)$. In addition, a somewhat similar
question, concerning the {\it min-entropy} of $a_1\cdot
f(\bA)+a_2\cdot\bA$ for random $a_1$ and $a_2$ in $\efq$ and for
large $q$ was implicitly considered in the merger 
literature,\footnote{The distribution of $a_1$ and $a_2$ depends on
whether the merger in question is the {\it linear} merger or the {\it
curve} merger, see, e.g., the introduction of \cite{DW}.  For example,
for the curve merger of \cite{DW}, it was shown in \cite{DKSS} that for any
$\varepsilon,\delta>0$, the weighted sum is
$\varepsilon$-close (in statistical distance) to having min-entropy
$(1-\delta)\cdot n\cdot \log(q)$, as long as $q\geq
(4/\varepsilon)^{1/\delta}$.} 
see, e.g., Sec. 3.1 of \cite{D}, and Theorem 18 of \cite{DKSS}.

The main contributions of the current note are the following two
theorems.

\begin{theorem}\label{thm:main}
Let $n\geq 1$ be an integer, let $f\colon\efq^n\to\efq^n$ be an
arbitrary function, and for $\bk\in 
\efq^n$, let $\gk\colon\efq^n\to\efq^n$ be defined by
$$
\gk(\bx):=f(\bx)+(k_1x_1,k_2x_2,\ldots,k_n x_n).
$$
Suppose that a random
variable $\bA$ is uniformly distributed over $\efq^n$.  Then
\begin{eqnarray*}
\frac{1}{q^n}\sum_{\bk\in \efq^n} H_2(\gk(\bA)) & \geq &\log(q^n) -
n\log\left(2-\frac{1}{q}\right).
\end{eqnarray*}
\end{theorem}

The point of the theorem is that the average over $\bk$ of
$H_2(\gk(\bA))$ is at most about $n$ bits below the entropy of a uniform
distribution over $\efq^n$, regardless of $q$ and $f$. Of course,
since the Shannon entropy is not smaller than the \renyis entropy, we
may replace $H_2$ by $H$ in Theorem \ref{thm:main}. In fact, 
a stronger result is proven: 

\begin{theorem}\label{thm:main2}
Using the terminology of Theorem \ref{thm:main}, we have
\begin{equation}\label{eq:main2}
\frac{1}{q^n}\sum_{\bk\in \efq^n} \cp(\gk(\bA)) 
\leq 
\frac{1}{q^n}\left(2-\frac{1}{q}\right)^n, 
\end{equation}
with equality if for all $i$, $f_i(\bx)$ depends only on $x_i$.
\end{theorem}

Note that by Jensen,
\begin{eqnarray*}
\frac{1}{q^n}\sum_{\bk\in \efq^n}H_2(\gk(\bA)) & = &
-\sum_{\bk}\frac{1}{q^n} \log\big(\cp(\gk(\bA))\big)\\
& \geq & -\log\left(\frac{1}{q^n}\sum_{\bk} \cp(\gk(\bA))\right),  
\end{eqnarray*}
and hence Theorem \ref{thm:main2} implies Theorem \ref{thm:main}.

As stated in the theorem itself, the bound of Theorem \ref{thm:main2}
is tight. The bound of Theorem \ref{thm:main} is also
tight, as seen by the following proposition.

\begin{proposition} \label{prop:tight}
For the function $f\colon\efq^n\to\efq^n$ defined by
$f(\bx):=(x_1^2,\ldots, x_n^2)$, we have (using the terminology of
Theorem \ref{thm:main})
$$
\frac{1}{q^n}\sum_{\bk\in \efq^n} H(\gk(\bA)) =
\log(q^n)-n\left(1-\frac{1}{q}\right),
$$
and
$$
\frac{1}{q^n}\sum_{\bk\in \efq^n} H_2(\gk(\bA)) = \begin{cases}
\log(q^n)-n\left(1-\frac{1}{q}\right) & \text{if $q$ is even,}\\
\log(q^n)-n\log\left(2-\frac{1}{q}\right) & \text{otherwise}.
\end{cases}
$$
\end{proposition}

\section{Proof of Theorem \ref{thm:main2}}

The proof begins as the proof of the Leftover Hash Lemma as
appearing in \cite{CF}. Letting $\bK$ and $\bA'$ be random variables
uniformly distributed over $\efq^n$ such that $\bA$, $\bK$ and $\bA'$
are jointly independent, the left-hand side of (\ref{eq:main2}) can be
written as  
\begin{eqnarray*}
\frac{1}{q^n}\sum_{\bk\in \efq^n} \cp(\gk(\bA)) & = & \sum_{\bk\in
\efq^n} \pr(\bK=\bk)\cdot \pr\big(\gk(\bA) = 
\gk(\bA')\big) \\
  & = & \sum_{\bk\in \efq^n} \pr(\bK=\bk)\cdot \pr\big(\gK(\bA) =
\gK(\bA')|\bK = \bk\big) \\
  & = & \pr\big(\gK(\bA) = \gK(\bA')\big).
\end{eqnarray*}

It follows that Theorem \ref{thm:main2} is an immediate consequence
of the following Lemma.

\begin{lemma}
Using the above notation,
$$
\pr\big(\gK(\bA) = \gK(\bA')\big)\leq 
\frac{1}{q^n}\left(2-\frac{1}{q}\right)^n
$$
with equality if for all $i$, $f_i(\bx)$ depends only on $x_i$. 
\end{lemma}

\begin{proof}
For $\bx,\bx'\in \efq^n$, let $\dham(\bx,\bx')$ be the Hamming  
distance between $\bx$ and $\bx'$ (number of coordinates $i$ for which
$x_i\neq x'_i$) and let $\bx\odot\bx':=(x_1 x'_1,\ldots, x_n
x'_n)$. We have 

\begin{eqnarray}
\pr\big(\gK(\bA) =  \gK(\bA')\big) & = &  \pr(\bA=\bA') + \sum_{d=1}^n
\pr(\dham(\bA,\bA')=d)\cdot\nonumber\\
&& \cdot \pr\big(\gK(\bA) = \gK(\bA')\big|\dham(\bA,\bA')=d\big). \label{eq:cont}
\end{eqnarray}

Now, 
$$
\pr\big(\gK(\bA) = \gK(\bA')\big|\dham(\bA,\bA')=d\big)
$$
(probability over $\bK$, $\bA$ and $\bA'$) is the average over pairs
of vectors $\bs{a},\bs{a'}\in\efq^n$ of Hamming 
distance $d$ of expressions like
\begin{equation}\label{eq:mid}
\pr \big(f(\ba)+ \bK\odot \ba = f(\ba')+ \bK\odot \ba'\big)
\end{equation}
(probability over $\bK$).
The last expression is either $0$ (if $f_i(\ba)\neq
f_i(\ba')$ for some $i$ for which $a_i=a'_i$),\footnote{Note that
this cannot happen if for all $i$, $f_i(\bx)$ depends only on
$x_i$. This will show that for such functions we have equality in the
proposition.} or $q^{n-d}/q^n$ otherwise ($d$ entries of $\bK$ are
determined by the equation, and the other $n-d$ entries are free). So,
in either case, the expression in (\ref{eq:mid}) is $\leq q^{-d}$
(with equality if for all $i$, $f_i$ depends only on the $i$th argument),
and hence so is the average of these expressions. Substituting in
(\ref{eq:cont}), we get 
\begin{eqnarray*}
\pr\big(\gK(\bA)=\gK(\bA')\big)  &\leq &
\cp(\bA) + \sum_{d=1}^n
\frac{q^n\binom{n}{d}(q-1)^d}{q^{2n}}q^{-d} \\
& = &\frac{1}{q^n} + \frac{1}{q^n}\sum_{d=1}^n
\binom{n}{d}\left(1-\frac{1}{q}\right)^d  \\
& = &\frac{1}{q^n} + \frac{1}{q^n}\left[\left(2-\frac{1}{q}\right)^n-1
\right]\\
& = & \frac{1}{q^n}\left(2-\frac{1}{q}\right)^n,
\end{eqnarray*}
with equality if for all $i$, $f_i(\bx)$ depends only on $x_i$.

\end{proof}

\section{Proof of proposition \ref{prop:tight}}
The assertion regarding the average Shannon entropy will follow immediately
from the chain rule for conditional Shannon entropy if
we prove that for $n=1$ and for the function $f\colon \efq\to\efq$
defined by $f(x)=x^2$, we have 
\begin{equation}\label{eq:one_enough}
\frac{1}{q}\sum_{k\in\efq}H(g_k(A)) = \log(q)-\left(1-\frac{1}{q}\right)
\end{equation}
for $A$ uniformly distributed on $\efq$.

Suppose first that $q$ is even. Then $g_0=(x\mapsto x^2)$ is a
permutation on $\efq$ (in fact, an automorphism), and so 
$H(g_0(A))=\log(q)$. For $k,y\in \efq$, let $\xkys:=\gks^{-1}(y)$. We
claim that for all $k\in\efq^*$ and for all $y\in 
\efq$ with $\xkys\neq \emptyset$, there are exactly 2 elements in
$\xkys$: On one hand,   
there are at most two solutions to a quadratic equation, and on the
other hand, for $x\in \xkys$, $x+k$ is different from $x$ and
satisfies $g_k(x+k)=g_k(x)$, which means that $x+k\in \xkys$. Hence
in the case of characteristic 2, the average entropy is
$(1/q)\cdot\log(q)+(1-1/q)\cdot\log(q/2)$, as desired.  

For odd $q$, we claim that for all $k\in\efq$, there is
a single $y$ with $|\xkys|=1$, and $(q-1)/2$ values of $y$ with
$|\xkys|=2$:  Fix $k$, take $y$ with $\xkys\neq\emptyset$, and let
$x\in \xkys$. Clearly, $g_k(-k-x)=g_k(x)$, and if $x\neq -k/2$, then
$-k-x\neq x$, which implies that $|\xkys|=2$. For $y$ with $-k/2\in
\xkys$, $|\xkys|$ must therefore be odd, and hence necessarily
equals\footnote{Of course, the last $y$ equals $-k^2/4$, and the
fact that $|\xkys|=1$ for this $y$ may also be verified directly.} 1.
Hence in the case of odd characteristic, the average entropy is
$((q-1)/2)\cdot(2/q)\cdot \log(q/2) + (1/q)\cdot\log(q)$, as in
(\ref{eq:one_enough}).   

It remains to calculate the average \renyis entropy for
$f=(\bx\mapsto(x_1^2,\ldots,x_n^2))$. It follows from the above
discussion on the Shannon entropy that if $q$ is even, then for all
$\bk$ and all $i$, the collision probability of the $i$-th entry of
$\gk(\bA)$ equals $2/q$ if $k_i\neq 0$ (uniform distribution on $q/2$
elements), and $1/q$ if $k_i=0$. As the collision probability of a
vector of jointly independent random variables is the product of the
individual collision probabilities, it follows that $\cp(g_{\bk}(\bA))
= 2^{\weight(\bk)}/q^n$, where $\weight(\bk)$ is the Hamming weight
of $\bk$ (number of nonzero coordinates in $\bk$). 

Since\footnote{One way to verify the following identity is to note
that the sum $W_q(n)$ of the weights of all vectors in $\efq^n$
satisfies $W_q(1)=q-1$ and
$W_q(n)=W_q(n-1)+(q-1)\cdot(W_q(n-1)+q^{n-1})$ for $n\geq 2$.}
$\sum_{\bk\in\efq^n}\weight(\bk) = nq^n-nq^{n-1}$,
we get
\begin{eqnarray*}
\frac{1}{q^n}\sum_{\bk}H_2(g_{\bk}(\bA)) & = &
\frac{1}{q^n}\sum_{\bk}(\log(q^n)-w(\bk)) \\ 
& = & \log(q^n) - \frac{1}{q^n}(nq^n-nq^{n-1}) \\ 
& = & \log(q^n) - n\left(1-\frac{1}{q}\right),
\end{eqnarray*}
as desired.

Finally, if $q$ is odd, then it follows from the discussion in the
beginning of the proof that for all $\bk$, the
collision probability of any entry of $\gk(\bA)$ equals
$$
\frac{1}{q^2}+\frac{q-1}{2}\bigg(\frac{2}{q}\bigg)^2 = \frac{2q-1}{q^2}.
$$
Because the collision probability of $\gk(\bA)$ is the product of the
collision probabilities of the individual entries, it
follows that for all $\bk$,
$$
H_2(g_{\bk}(\bA)) = -\log\left(\frac{1}{q^{2n}}\cdot(2q-1)^n\right)
= -\log\left(\frac{1}{q^n}\cdot\left(2-\frac{1}{q}\right)^n\right),
$$
which completes the proof.

\begin{remark}
{\rm
Note that in Proposition \ref{prop:tight}, the components $f_i$ may be
any quadratic functions $x_i\mapsto a_ix_i^2+b_ix_i+c_i$ with $a_i\neq
0$ for all $i$
(eliminating $a_i$ and $c_i$ is done by an invertible function, and
then the linear term is ``absorbed'' in the averaging over $k_i$).  
}
\end{remark}

\section*{Acknowledgments} 
We are grateful to Avner Dor for carefully reading several earlier
drafts and for his helpful comments. We would also like to thank Simon
Litsyn for pointing us to \cite{Car}.



\begin{thebibliography}{99}

\bibitem{Car} L.~Carlitz, {\em On the number of distinct values of a
polynomial with coefficients in a finite field}, {\em Proc. Japan
Acad.}, 31, pp. 119--120, 1955.

\bibitem{CF} R.~Cramer and S.~Fehr, {\em The mathematical theory of
information, and applications}, ver. 2.0. Course notes  available
online at http://homepages.cwi.nl/~bouman/icc/InfTheory2.pdf

\bibitem{D} Z.~Dvir, {\em From randomness extraction to rotating
needles}, ECCC TR09-077, 2009. 

\bibitem{DKSS} Z.~Dvir, S.~Kopparty, S.~Saraf, and M.~Sudan,
{\em Extensions to the method of multiplicities, with applications to
Kakeya sets and mergers}, arXiv:0901.2529v2.

\bibitem{DW} Z.~Dvir and A.~Wigderson,
\emph{Kakeya sets, new mergers and old extractors},
in {\em Proc. FOCS 2008}, pp. 625--633.

\bibitem{KLSS} S.~Kopparty, V.~F.~Lev, S.~Saraf, and M.~Sudan,
\emph{Kakeya-type sets in finite vector spaces}, arXiv:1003.3736v1.

\bibitem{Sti} D.~R.~Stinson, {\em Universal hash families and the
leftover hash lemma, and applications to cryptography and computing}, 
{\it J. Combin. Math. Combin. Comput.}, 42, pp. 3-31, 2002.

\bibitem{Zamir} R.~Zamir, {\em Anti-structure problems}, in {\em
Proc. Int. Zurich Seminar on Communications}, Feb. 29 -- Mar. 2, 2012,
pp. 91--94; available also as arXiv:1109.0414v1. 

\end{thebibliography}
\end{document}